\documentclass[11pt, a4paper]{article}
\usepackage[latin1]{inputenc}
\usepackage[T1]{fontenc}
\usepackage[english]{babel}
\usepackage{amssymb, amsmath, amsthm}
\usepackage[colorlinks=true, allcolors=blue]{hyperref}
\usepackage{geometry}
\usepackage{fancyhdr}
\usepackage{relsize}
\usepackage{upgreek}
\usepackage{verbatim}
\usepackage{enumitem}
\usepackage{calrsfs}
\usepackage{todonotes}
\usepackage{tikz}
\usetikzlibrary{matrix,arrows,calc,cd}
\usepackage{package}
\makeatletter
\renewcommand*{\@fnsymbol}{\@arabic}
\makeatother
\title{On the Cohomological Hall Algebra of the Kronecker Quiver}
\author{H. Franzen%
\thanks{Faculty of Mathematics, Ruhr-Universit\"at Bochum, Universit\"atsstra{\ss}e 150, 44780 Bochum\newline \href{mailto:hans.franzen@rub.de}{hans.franzen@rub.de}}
\and M. Reineke%
\thanks{Faculty of Mathematics, Ruhr-Universit\"at Bochum, Universit\"atsstra{\ss}e 150, 44780 Bochum\newline \href{mailto:markus.reineke@rub.de}{markus.reineke@rub.de}}%
}

\date{}
\begin{document}

	\maketitle
	\begin{abstract}
		\noindent
		We give a short introduction to Cohomological Hall algebras of quivers and describe the semistable Cohomological Hall algebra of central slope of the Kronecker quiver in terms of generators and relations.
	\end{abstract}
	
	\section{Generalities on Cohomological Hall algebras}
	
	The Cohomological Hall algebra was introduced by Kontsevich and Soibelman in \cite{KS:11}. There are various levels of generality in which a Cohomological Hall algebra can be defined. There is the critical Cohomological Hall algebra of a quiver with a potential which is also introduced in \cite{KS:11} and has been studied for example in \cite{Davison:17}, \cite{SV:17:1, SV:17:2}, \cite{YZ:18:1}. There are also versions of the Cohomological Hall algebra of sheaves on a curve or a surface, as studied in \cite{KSV:17} or of Higgs sheaves \cite{SS:18}.
	
	In this work, we concentrate on (semistable) Cohomological Hall algebras of quivers. The aim of this section is to recall their definition as well as some structural results.
	
	\subsection{Definition of the Cohomological Hall algebra of a quiver}
	
Let $Q$ be a finite quiver with set of vertices $Q_0$ and set of arrows $Q_1$. For an arrow $\alpha \in Q_1$ the source and target of $\alpha$ are denoted by $s(\alpha)$ and $t(\alpha)$, respectively. For a dimension vector $\mathbf{d}\in\Lambda^+=\mathbb{N}Q_0\subseteq\Lambda=\mathbb{Z}Q_0$, we define 
$$
	R_\mathbf{d} = \bigoplus_{\alpha \in Q_1} \Hom(\C^{d_{s(\alpha)}},\C^{d_{t(\alpha)}})
$$ 
which we regard as an affine space. Its points correspond to complex representations of $Q$ of dimension vector $\mathbf{d}$. The linear algebraic group $G_\mathbf{d} = \prod_{i \in Q_0} \GL_{d_i}(\C)$ acts on $R_\mathbf{d}$ via base change, so that the orbits for this action correspond to the isomorphism classes of representations of dimension vector $\mathbf{d}$. On the $\Lambda_+$-graded $\Q$-vector space of equivariant cohomology groups
$$
	H(Q)=\bigoplus_{\mathbf{d}\in\Lambda^+}H^*_{G_\mathbf{d}}(R_\mathbf{d};\Q)
$$ 
a product, denoted $*$, can be defined using pull-back and push-forward operations. Together with this product, $H(Q)$ is an associative algebra, naturally graded by $\Lambda^+$. It is called the Cohomological Hall algebra. The product resembles the convolution product of the Ringel-Hall algebra of a quiver of \cite{Ringel:90}. As we will not use the definition of this convolution-like product we refer to \cite[2.2]{KS:11} for details and instead recall the algebraic description of the multiplication of \cite[Thm.\ 2]{KS:11}. The homogeneous component of degree $\mathbf{d}$ of the algebra $H(Q)$ is given by
$$
	H(Q)_\mathbf{d}\cong \bigotimes_{i\in Q_0}\mathbb{Q}[x_{i,1},\ldots,x_{i,d_i}]^{S_{d_i}},
$$
and the multiplication is given by a shuffle product with kernel: 
For $f=f(x_{i,k}\, :\, i\in Q_0,\, k\leq d'_i)\in H(Q)_\mathbf{d'}$ and $g=g(x_{i,k}\, :\, i\in Q_0,\, k\leq d''_i)\in H(Q)_\mathbf{d''}$,
the product $(f*g)(x_{i,k}\, :\, i\in Q_0,\, k\leq d'_i+d''_i)$ equals
$$
	\sum_{(\sigma_i)}f(x_{i,\sigma_i(k)},\, i\in Q_0,\, k\leq d'_i)g(x_{i,\sigma_i(d'_i+l)},\, i\in Q_0,\, l\leq d''_i)\prod_{i,j\in Q_0}\prod_{k=1}^{d'_i}\prod_{l=1}^{d''_j}(x_{j,\sigma_j(d'_j+l)}-x_{i,\sigma_i(k)})^{-\langle \mathbf{e}_i,\mathbf{e}_j\rangle},
$$
where the sum ranges over all tuples $(\sigma_i)_{i\in Q_0}$ such that $\sigma_i$ is a $(d'_i,d''_i)$-shuffle permutation. In the above formula $\langle \blank,\blank \rangle$ denotes the Euler form of $Q$. This is the $\Z$-bilinear form on $\Z^{Q_0}$ given by $\langle \mathbf{a},\mathbf{b} \rangle = \sum_{i \in Q_0} a_ib_i - \sum_{\alpha \in Q_1} a_{s(\alpha)}b_{t(\alpha)}$. Moreover $\mathbf{e}_i \in \Z^{Q_0}$ is the $i$\textsuperscript{th} coordinate vector.

\subsection{Symmetric case}

The Cohomological Hall algebra has particularly nice properties if the quiver $Q$ is symmetric, which means that its Euler form is a symmetric bilinear form. 
In this case, the algebra $H(Q)$ is graded by the semi-group $\Gamma = \Lambda^+ \times \Z$ by defining
$$
	H(Q)_{(\mathbf{d},k)} = H_{G_d}^{k-\langle \mathbf{d},\mathbf{d} \rangle}(R_\mathbf{d}).
$$
It is possible to find a bilinear form $\psi$ on $\Z/2\Z^{Q_0}$ in such a way that the twisted multiplication $f \star g := (-1)^{\psi(\mathbf{d'},\mathbf{d''})} f*g$ is graded commutative in the sense that $f \star g = (-1)^{kl} g \star f$ for $f \in H(Q)_{(\mathbf{d'},k)}$ and $g \in H(Q)_{(\mathbf{d''},l)}$, see \cite[2.6]{KS:11}.
Efimov proves in \cite{Efimov:12} a conjecture of Kontsevich and Soibelman that $H(Q)$, equipped with the twisted multiplication, is isomorphic to the graded symmetric algebra $\mathrm{Sym}^*(V)$ for a $\Gamma$-graded vector space $V$ which is of the form $V = V^{\mathrm{prim}} \otimes \Q[z]$, where $z$ lives in degree $(0,2)$. The Poincar\'e series of $V^{\mathrm{prim}}$ encodes the motivic Donaldson-Thomas invariants of $Q$. In fact, this was the first proof of the positivity conjecture for Donaldson--Thomas invariants of a symmetric quiver.

\subsection{Semistable Cohomological Hall algebra}

For general $Q$, the structure of $H(Q)$ can be analyzed using semistable Cohomological Hall algebras: 
We choose a stability function $\Theta\in\Lambda^*$. For $\mathbf{d} \in \Lambda^+\setminus \{0\}$ the rational number $\mu(\mathbf{d}) = \Theta(d)/\smash{\sum_{i \in Q_0}} d_i$ is called the slope of $\mathbf{d}$. A representation $M$ of $Q$ is called $\Theta$-semistable (resp.\ $\Theta$-stable) if $\mu(\mathbf{dim}\, M') \leq \mu(\mathbf{dim}\, M)$ (resp.\ $\mu(\mathbf{dim}\, M') < \mu(\mathbf{dim} M)$) for every non-zero proper subrepresentation $M' \subseteq M$.
Let $R_\mathbf{d}^{\Theta-\mathrm{sst}}\subseteq R_\mathbf{d}$ be the Zariski-open subset of $\Theta$-semistable points. For a fixed $\mu\in\mathbb{Q}$, we consider the set $\Lambda^+_\mu = \{\mathbf{d} \in \Lambda^+\setminus \{0\} \mid \mu(\mathbf{d})=\mu\} \cup \{0\}$. We then define
$$
H^{\Theta}_\mu(Q)=\bigoplus_{\mathbf{d}\in\Lambda^+_\mu}H^*_{G_\mathbf{d}}(R^{\Theta-\mathrm{sst}}_\mathbf{d};\Q)
$$
with the convolution structure induced by the one on $H(Q)$, called the $\Theta$-semistable Cohomological Hall algebra of $Q$ for slope $\mu$.
It is shown in \cite[Thm.\ 6.2]{FR:18:Sst_ChowHa} that the Harder-Narasimhan filtration of representations of $Q$ induces an isomorphism of $\Lambda^+$-graded vector spaces
$$
	H(Q)\cong \bigotimes^{\rightarrow}_{\mu\in\mathbb{Q}}H_\mu^{\Theta}(Q).
$$

In case the Euler form $\langle\_,\_\rangle_Q$ of $Q$ is symmetric when restricted to $\Lambda^+_\mu$, good structural properties of $H^\Theta_\mu(Q)$ are expected; however, even in simple cases it is no longer (graded) commutative. In \cite[Thm.\ C]{DM:16}, Davison and Meinhardt prove (in a more general context) a PBW type theorem for $H^\Theta_\mu(Q)$. They show the existence of a filtration (the perverse filtration) on $H^\Theta_\mu(Q)$ whose associated graded is a graded symmetric algebra.

\subsection{Generators}

There are several geometric interpretations of the generators of the Cohomological Hall algebra. For instance, Chen identifies in \cite{Chen:14} the primitive part of the Cohomological Hall algebra of a symmetric quiver with the invariant part under a Weyl group action of the cohomology of a variety introduced in \cite{HLV:13}. Another interpretation can be given as follows. Denoting by $R_\mathbf{d}^{\Theta-\mathrm{st}}(Q)$ the locus of stable representations, it is shown in \cite[Thm.\ 9.1]{FR:18:Sst_ChowHa} that
$$
	A^*_{G_\mathbf{d}}(R_\mathbf{d}^{\Theta-\mathrm{st}}(Q))_\mathbb{Q}\cong H^\Theta_\mathbf{d}(Q)/\sum H^\Theta_\mu(Q)_\mathbf{d'}* H^\Theta_\mu(Q)_\mathbf{d''},
$$
where the sum ranges over all decompositions $\mathbf{d}=\mathbf{d'}+\mathbf{d''}$ such that $\mu(\mathbf{d})=\mu(\mathbf{d''})$. This describes the quotient of the algebra $H^\Theta_\mu(Q)$ by the square of the augmentation ideal $\smash{\bigoplus_{\mathbf{d} \in \Lambda_\mu^+ \setminus \{0\}}} H^\Theta_\mu(Q)_\mathbf{d}$ as the equivariant Chow ring $A_{G_d}^*(\smash{R_\mathbf{d}^{\Theta-\mathrm{st}}}(Q))_\Q$  of the stable locus in the sense of \cite{EG:98}.

\subsection{Presentation}\label{ssp}

The key to describing the semi-stable Cohomological Hall algebra explicitly in terms of generators and relations (which will be worked out in the case of the Kronecker quiver in the following sections) is the following tautological presentation for $H^\Theta_\mu(Q)$:
$$
	H^\Theta_\mu(Q)_\mathbf{d}\cong H(Q)_\mathbf{d}/\sum H(Q)_\mathbf{d'}* H(Q)_\mathbf{d''},
$$
where the sum ranges over all decompositions $\mathbf{d}=\mathbf{d'}+\mathbf{d''}$ such that $\mu(\mathbf{d'})>\mu(\mathbf{d''})$.

\subsection{Examples}

There are three examples in which an explicit description of the Cohomological Hall algebra in terms of generators and relations is known, namely for the quivers $A_1$ (one vertex and no arrows), $L_1$ (one vertex, one arrow), and $\tilde{A}_{2,0}$ (two vertices and two arrows of opposite orientation). The first two examples were already discussed in \cite[2.5]{KS:11}, the latter in \cite[10.1]{FR:18:Sst_ChowHa}. All these quivers are symmetric whence we know by Efimov's theorem that their Cohomological Hall algebra is a graded symmetric algebra.

For $Q = A_1$, the Cohomological Hall algebra as a vector space is $H(A_1) = \bigoplus_{d \geq 0} \Q[x_1,\ldots,x_d]^{S_d}$. The Euler form of $A_1$ is $\langle d,e \rangle = de$. The $\Lambda^+ \times \Z=\N \times \Z$-grading of $H(A_1)$ is hence given by assigning to each $x_i$ the degree 2 and letting
$$
	H(A_1)_{(d,k)} = \Q[x_1,\ldots,x_d]^{S_d}_{k-d^2} = \bigoplus_{2n_1 + \ldots + 2dn_d = k-d^2} \Q\cdot e_1^{n_1}\ldots e_d^{n_d}.
$$
The multiplication of $H(A_1)$ is given by
$$
	(f*g)(x_1,\ldots,x_d) = \sum f(x_{\sigma(1)},\ldots,x_{\sigma(d')})g(x_{\sigma(d'+1)},\ldots,x_{\sigma(d)})\frac{1}{\prod_{r=1}^{d'}\prod_{s=1}^{d''} (x_{\sigma(d'+s)} - x_{\sigma(r)})}.
$$
The sum ranges over all $(d',d'')$-shuffle permutations, $f \in H(A_1)_{d'}$, $g \in H(A_1)_{d''}$, and $d=d'+d''$. We see that this multiplication is anti-commutative, whence $f * f = 0$. This induces an algebra homomorphism $\bigwedge^* H(A_1)_1 \to H(A_1)$. Let $\psi_i$ be the polynomial $\psi_i(x) = x^i \in H(A_1)_{(1,2i+1)}$ (the $i$ in the exponent refers to the usual multiplication of polynomials and not to the multiplication in the Cohomological Hall algebra). The elements $\psi_0,\psi_1,\psi_2,\ldots$ form a homogeneous basis of $H(A_1)$. It is easy to see that for integers $0 \leq k_1 < \ldots k_d$ we get
$$
	\psi_{k_1} * \ldots * \psi_{k_d} = s_\lambda(x_1,\ldots,x_d)
$$
where $\lambda = (k_d-d+1,\ldots,k_2-1,k_1)$ and $s_\lambda$ is the Schur function associated with the partition $\lambda$. This shows that the induced homomorphism $\bigwedge^*(\psi_0,\psi_1,\psi_2,\ldots) \to H(A_1)$ is surjective and a comparison of the generating functions of these two algebras yields that the surjection must be an isomorphism. So
$$
	H(A_1) \cong \operatorname{Sym}^*(\Q(1,1)\otimes\Q[z])
$$
as $\N \times \Z$-graded algebras, where $\Q(d,i)$ denotes a one-dimensional vector space in degree $(d,i)$ and $z$ is an element of degree $(0,2)$.

The Cohomological Hall algebra of the quiver $Q=L_1$ with one loop is as a vector space again $H(L_1) = \bigoplus_{d \geq 0} \Q[x_1,\ldots,x_d]^{S_d}$ but as the Euler form vanishes the $\N \times \Z$-grading is given by
$$
	H(A_1)_{(d,k)} = \Q[x_1,\ldots,x_d]^{S_d}_{k} = \bigoplus_{2n_1 + \ldots + 2dn_d = k} \Q\cdot e_1^{n_1}\ldots e_d^{n_d}.
$$
The multiplication is given by
$$
	(f*g)(x_1,\ldots,x_d) = \sum f(x_{\sigma(1)},\ldots,x_{\sigma(d')})g(x_{\sigma(d'+1)},\ldots,x_{\sigma(d)}).
$$
The multiplication is commutative which shows that we obtain a homomorphism $S^*(H(L_1)_1) \to H(L_1)$. Again we denote $\psi_i(x) = x^i \in H(L_1)_{(1,2i)}$. We can see that for integers $k_1 \geq \ldots \geq k_d \geq 0$
$$
	\psi_{k_1} * \ldots * \psi_{k_d} = c_\lambda m_\lambda(x_1,\ldots,x_d)
$$
where $m_\lambda$ is the monomial symmetric function associated with the partition $\lambda = (k_1,\ldots,k_d)$ and $c_\lambda$ is a positive integer. This shows surjectivity of $S^*(\psi_0,\psi_1,\psi_2,\ldots) \to H(L_1)$. As the generating series of these algebras agree the map is an isomorphism. Therefore
$$
	H(L_1) \cong \operatorname{Sym}^*(\Q(1,0) \otimes \Q[z])
$$
as $\N \times \Z$-graded algebras. Again $z$ is of degree $(0,2)$.

For the quiver $\tilde{A}_{2,0}$ of affine type $\tilde{A}_1$ with the cyclic orientation, an explicit presentation can be determined by decomposing it into its semistable parts. the choice of an appropriate stability condition yields a tensor product decomposition (as $\Lambda^+ \times \Z$-graded algebras) into three factors, two of which are isomorphic to the Cohomological Hall algebra of $A_1$ and the other being isomorphic to $H(L_1)$. Concretely
$$
	H(\tilde{A}_{2,0}) \cong \operatorname{Sym}^*\Big( \big(\Q((1,0),1) \oplus \Q((1,1),0) \oplus \Q((0,1),1) \big) \otimes \Q[z]\Big)
$$
as $\Lambda^+ \times \Z$-graded algebras and $z$ of degree $(0,2)$.


\section{The case of the Kronecker quiver}

We now specialize the above results to the Kronecker quiver $Q=K_2$ with vertices $i$ and $j$ and two arrows from $i$ to $j$ for the stability $\Theta(d_i,d_j)=d_i$, recall some further results from \cite[Section 10.2.]{FR:18:Sst_ChowHa}, and state our main result.

It is known that $$H^\Theta_{\mu(\mathbf{d})}(K_2)\cong\Lambda^*(\Q(d,1) \otimes \mathbb{Q}[z])$$ is isomorphic to a countably generated exterior algebra for $\mathbf{d}$ of the form $(n,n+1)$ or $(n+1,n)$ for $n\geq 0$, and that $H^\Theta_\mu(K_2)$ is zero if $\mu$ is neither zero nor of the form $n/(2n\pm 1)$ for $n\geq 1$.

It thus remains to determine the structure of $H^\Theta_{1/2}(K_2)$. We denote this algebra simply by $A$; 
it is $\mathbb{N}$-graded by
$A_n=\smash{H^\Theta_{1/2}(K_2)_{(n,n)}}$. It is already known that there is an isomorphism of $\mathbf{N}$-graded vector spaces
$$A\cong\mathrm{Sym}^*(A^*(\mathbb{P}^1)[z])$$
with the generating vector space in degree $1$. Our main result is:

\begin{thm}\label{t1} The algebra $A$ is generated in degree $1$ by generators $e_n$ for $n\geq 0$ and $f_n$ for $n\geq 1$, subject to the relations
\begin{align*}
[E(X),E(Y)]_*&=2(Y-X)\frac{YE(Y)-XE(X)}{Y-X}*\frac{YF(Y)-XF(X)}{Y-X},\\
[E(X),F(Y)]_*&=(Y-X)\left(\frac{YF(Y)-XF(X)}{Y-X}\right)^{*2},\\
[F(X),F(Y)]_*&=0
\end{align*}
in terms of generating series
\begin{align*}
E(X)&=\sum_{n\geq 0}e_nX^n,\;\;\; F(X)=\sum_{n\geq 0}f_{n+1}X^n.
\end{align*}
More explicitly, these identities are equivalent to the following relations:
\begin{align*}
[e_p,e_q]_*&=2(e_p*f_q+e_{p+1}*f_{q-1}+\ldots+e_{q-1}*f_{p+1}) &\mbox{ for }p<q,\\
[e_p,f_{q+1}]_*&=f_{p+1}*f_q+\ldots+f_qf_{p+1} &\mbox{ for }p<q,\\
[e_p,f_{q+1}]_*&=-f_{q+1}*f_p-\ldots- f_p*f_{q+1} &\mbox{ for }p>q,\\
[f_{p+1},f_{q+1}]_*&=0 &\mbox{ for }p<q.
\end{align*}
\end{thm}

From these defining relation we see that $A$ is a filtered algebra, with the filtration induced by the degree function $\deg(e_i)=1$ and $\deg(f_i)=0$ for all $i$, such that the associated graded is the symmetric algebra of $A_1$.

\section{Proof of Theorem \ref{t1}}

\subsection{Idea of the proof}

Since $R_{(n,n)}^{\Theta-\mathrm{st}}(K_2)=\emptyset$ for $n\geq 2$, the tautological presentation of Section \ref{ssp} yields that
$$A_n=\sum_{k+l=n}A_k* A_l$$
for $n\geq 2$. By induction over $n$, we find that $A_n$ is generated in degree $1$. We will prove the following property in the next section:

There exists an ordered basis $(b_i)_{i\in I}$ for $A_1$ such that
$$b_j*b_i=\sum_{k\leq l}c^{i,j}_{k,l}\cdot b_k* b_k\mbox{ for all }i<j.$$

This implies that every $A_n$ is linearly generated by the monomials $b_{i_1}\cdot\ldots\cdot b_{i_n}$ for $i_1\leq\ldots\leq i_n$ in $I$. By the above result, this implies that these monomials already form a linear basis of $A_n$. This in turn implies that $A$ is generated by the $b_i$ subject to the relations

$$b_j* b_i-\sum_{k\leq l}c^{i,j}_{k,l}\cdot b_k*b_l\mbox{ for all }i<j.$$

\subsection{Computations}

To compute relations in $A$, we first describe some of the multiplication maps in $A(K_2)$ more explicitly using the above algebraic description as a modified shuffle product:

As above, we identify $$H(K_2)_{(d,e)}\cong\mathbb{Q}[x_1,\ldots,x_d]^{S_d}\otimes\mathbb{Q}[y_1,\ldots,y_e]^{S_e}.$$

\begin{lem} We have:
\begin{itemize}
\item The multiplication map $H(K_2)_{(1,0)}\otimes H(K_2)_{(0,1)}\rightarrow H(K_2)_{(1,1)}$ is given by
$$(f*g)(x,y)=f(x)g(y)(y-x)^2.$$
\item The multiplication map $H(K_2)_{(1,0)}\otimes H(K_2)_{(1,2)}\rightarrow H(K_2)_{(2,2)}$ is given by
\begin{align*}
(f*g)(x_1,x_2,y_1,y_2)=\frac{1}{x_2-x_1}\cdot\left(\right.& f(x_1)g(x_2,y_1,y_2)(y_1-x_1)^2(y_2-x_1)^2-\\
&\left. -f(x_2)g(x_1,y_1,y_2)(y_1-x_2)^2(y_2-x_2)^2\right).
\end{align*}
\item The multiplication map $H(K_2)_{(1,1)}\otimes H(K_2)_{(1,1)}\rightarrow H(K_2)_{(2,2)}$ is given by
\begin{align*}
(f*g)(x_1,x_2,y_1,y_2)=&\frac{1}{(x_2-x_1)(y_2-y_1)}\cdot\\
&\cdot\left(f(x_1,y_1)g(x_2,y_2)(y_2-x_1)^2-f(x_1,y_2)g(x_2,y_1)(y_1-x_1)^2\right.\\
&\ \ \left. -f(x_2,y_1)g(x_1,y_2)(y_2-x_2)^2+f(x_2,y_2)g(x_1,y_1)(y_1-x_2)^2\right).
\end{align*}
\end{itemize}
\end{lem}

We thus find $A_1\cong\mathbb{Q}[x,y]/(y-x)^2$, for which we choose the basis $e_n=x^n$ for $n\geq 0$ and $f_n=x^{n-1}(y-x)$ for $n\geq 1$.
We define generating series
\begin{align*}
E(X)&=\sum_{n\geq 0}e_nX^n=\frac{1}{1-xX}\in A_1[X],& F(X)&=\sum_{n\geq 0}f_{n+1}X^n=\frac{y-x}{1-xX}\in A_1[X].
\end{align*}

\begin{lem} The following relations hold in $A_2[X,Y]$ (where $X$ and $Y$ are viewed as variables commuting with each other and with $A_1$ and $A_2$):
\begin{align*}
[E(X),E(Y)]_*&=2\frac{(YE(Y)-XE(X))*(YF(Y)-XF(X))}{Y-X},\\
[E(X),F(Y)]_*&=\frac{(YF(Y)-XF(X))^{*2}}{Y-X},\\
[F(X),F(Y)]_*&=0.\\
\end{align*}
\end{lem}

\begin{proof} We first note that
$$\frac{YE(Y)-XE(X)}{Y-X}=\frac{1}{(1-xX)(1-xY)},$$
and similarly for the series $F$.

We have
$$\frac{1}{1-xX}*\frac{1}{1-xY}=$$
$$=\frac{1}{(x_2-x_1)(y_2-y_1)}\cdot\left( \frac{(y_2-x_1)^2}{(1-x_1X)(1-x_2Y)}-\frac{(y_1-x_1)^2}{(1-x_1X)(1-x_2Y)}\right.$$
$$\left. -\frac{(y_2-x_2)^2}{(1-x_2X)(1-x_1Y)}+\frac{(y_1-x_2)^2}{(1-x_2X)(1-x_1Y)}\right)=$$
$$=\frac{1}{x_2-x_1}\cdot\left( \frac{y_1+y_2-2x_1}{(1-x_1X)(1-x_2Y)}-\frac{y_1+y_2-2x_2}{(1-x_2X)(1-x_1Y)}\right).$$
Exchanging the variables $X$ and $Y$, we find
$$\frac{1}{1-xY}*\frac{1}{1-xX}=\frac{1}{x_2-x_1}\cdot\left( \frac{y_1+y_2-2x_1}{(1-x_2X)(1-x_1Y)}-\frac{y_1+y_2-2x_2}{(1-x_1X)(1-x_2Y)}\right).$$
Thus
\begin{align*}
\left[\frac{1}{1-xX},\frac{1}{1-xY}\right]_*&=\frac{2(y_1+y_2-x_1-x_2)}{x_2-x_1}\cdot\left(\frac{1}{1-x_1X)(1-x_2Y)}-\frac{1}{(1-x_2X)(1-x_1Y)}\right)\\
&=\frac{2(y_1+y_2-x_1-x_2)(Y-X)}{(1-x_1X)(1-x_2X)(1-x_1Y)(1-x_2Y)}.
\end{align*}

On the other hand, we have
\begin{align*}
&\frac{YE(Y)-XE(X)}{Y-X}*\frac{YF(Y)-XF(X)}{Y-X}\\
=&\frac{1}{(1-xX)(1-xY)}*\frac{y-x}{(1-xX)(1-xY)}\\
=&\frac{(y_2-x_2)(y_2-x_1)^2-(y_1-x_2)(y_1-x_1)^2-(y_2-x_1)(y_2-x_2)^2+(y_1-x_1)(y_1-x_2)^2}{(x_2-x_1)(y_2-y_1)(1-x_1X)(1-x_1Y)(1-x_2X)(1-x_2Y)}\\
=&\frac{y_1+y_2-x_1-x_2}{(1-x_1X)(1-x_1Y)(1-x_2X)(1-x_2Y)},
\end{align*}
from which the first claimed relation follows.

The second relation is proved similarly, so we omit some intermediate steps. On the one hand, we have
$$[E(X),F(Y)]_*=\left[\frac{1}{1-xX},\frac{y-x}{1-xY}\right]_*=\frac{(y_1+y_2-x_1-x_2)^2(Y-X)}{(1-x_1X)(1-x_2X)(1-x_1Y)(1-x_2Y)},$$
and on the other hand
$$(\frac{YF(Y)-XF(X)}{Y-X})^{*2}=(\frac{y-x}{(1-xX)(1-yY)})^{*2}=\frac{(y_1+y_2-x_1-x_2)^2}{(1-x_1X)(1-x_2X)(1-x_1Y)(1-x_2Y)},$$
and the claimed relation follows.

To prove the third relation, we first proceed similarly to the previous computations to arrive at
$$[F(X),F(Y)]_*=\frac{(y_1+y_2-x_1-x_2)((x_1+x_2)(y_1+x_2)-x_1^2-x_2^2-2y_1y_2)(Y-X)}{(1-x_1X)(1-x_2X)(1-x_1Y)(1-x_2Y)}.$$
This element is non-zero in $A_2[X,Y]=H_{(2,2)}(K_2)[X,Y]$, but we claim that it belongs to the image of the multiplication map
$$H(K_2)_{(1,0)}[X,Y]\otimes H(K_2)_{(1,2)}[X,Y]\rightarrow H(K_2)_{(2,2)}[X,Y].$$
Namely, by the above description of this map, we find
$$\frac{1}{(1-xX)(1-xY)}*\frac{1}{(1-xX)(1-xY)}=\frac{(y_1+y_2-x_1-x_2)(x_1^2+x_2^2+2y_1y_2-(x_1+x_2)(y_1+y_2))}{(1-x_1X)(1-x_1Y)(1-x_2X)(1-x_2Y)},$$
as claimed. The lemma is proved.
\end{proof}

We would now like to understand which relations between the individual basis elements $e_n$, $f_n$ result from the above identities of generating series. We have
$$[E(X),E(Y)]_*=\sum_{p,q\geq 0}[e_p,e_q]_*X^pY^q$$
and
$$\frac{YE(Y)-XE(X)}{Y-X}=\sum_{p,q\geq 0}e_{p+q}X^pY^q.$$
Thus
$$\frac{YE(Y)-XE(X)}{Y-X}*(YF(Y)-XF(X))=$$
$$=\sum_{i,j\geq 0}e_{i+j}X^iY^j*(\sum_{k\geq 0}f_{k+1}Y^{k+1}-\sum_{k\geq 0}f_{k+1}X^{k+1})=$$
$$=\sum_{p,q\geq 0}(\sum_{j+k=q-1}e_{p+j}f_{k+1}-\sum_{i+k=p-1}e_{i+q}f_{k+1})X^pY^q.$$
Thus for $p<q$, we find the relations
$$[e_p,e_q]_*=2(e_p*f_q+e_{p+1}*f_{q-1}+\ldots+e_{q-1}*f_{p+1}).$$
The same calculation for the second type of relation leads to
$$[e_p,f_{q+1}]_*=f_{p+1}*f_q+\ldots+f_q*f_{p+1}\mbox{ for }p<q,$$
$$[e_p,f_{q+1}]_*=-f_{q+1}*f_p-\ldots- f_p*f_{q+1}\mbox{ for }p>q.$$
Finally, we have
$$[f_{p+1},f_{q+1}]_*=0$$
for all $p<q$.
We can now finish the proof of the theorem: ordering the basis elements as
$$e_0,e_1,e_2,\ldots, f_1,f_2,\ldots,$$
we see that the above relations allow us to reorder any given product of the generators into a monomial in standard ordering.

\section{Further descriptions}

In this section, we indicate three alternative descriptions of the algebra $A$ in order to better understand the nature of its defining relations.

\subsection{Twisted symmetric algebra}

For a vector space $V$ and an operator $c\in{\rm End}(V\otimes V)$ such that $c^2={\rm id}$, we can define the $c$-twisted symmetric algebra
$${\rm Sym}^c(V)=T(V)/({\rm Ker}(c-{\rm id})).$$
In the case of the flip operator $c(v\otimes w)=w\otimes v$, we get the usual symmetric algebra. If $c$ moreover satisfies the quantum Yang-Baxter equation
$$(c\otimes {\rm id})({\rm id}\otimes c)(c\otimes{\rm id})=({\rm id}\otimes c)(c\otimes{\rm id})({\rm id}\otimes c),$$
the algebra ${\rm Sym}^c(V)$ is called a braided symmetric algebra.

\begin{thm}\label{symc} The algebra $A$ is isomorphic to ${\rm Sym}^c(A_1)$ for the operator $c$ defined on the formal series $E(X)$ and $F(X)$ by
$$c(E(X)\otimes E(Y))=E(Y)\otimes E(X)+(Y-X)(\frac{YE(Y)-XE(X)}{Y-X}\otimes\frac{YF(Y)-XF(X)}{Y-X}+$$
$$+\frac{YF(Y)-XF(X)}{Y-X}\otimes\frac{YE(Y)-XE(X)}{Y-X}),$$
$$c(E(X)\otimes F(Y))=F(Y)\otimes E(X)+(Y-X)\frac{YF(Y)-XF(X)}{Y-X}\otimes\frac{YF(Y)-XF(X)}{Y-X},$$
$$c(F(X)\otimes E(Y))=E(Y)\otimes F(X)+(Y-X)\frac{YF(Y)-XF(X)}{Y-X}\otimes\frac{YF(Y)-XF(X)}{Y-X},$$
$$c(F(X)\otimes F(Y))=F(Y)\otimes F(X).$$
\end{thm}

\proof We abbreviate $E(X,Y)=YE(Y)-XE(X)$, and similarly $F(X,Y)$. Then
$$c^2(E(X)\otimes F(Y))=c(F(Y)\otimes E(X))+\frac{1}{(Y-X)}c(F(X,Y)\otimes F(X,Y))=$$
$$=E(X)\otimes F(Y)+\frac{1}{X-Y}F(X,Y)\otimes F(X,Y)+\frac{1}{Y-X}F(X,Y)\otimes F(X,Y)=E(X)\otimes F(Y).$$
With a similar computation we find $c(E(X,Y)\otimes F(X,Y))=F(X,Y)\otimes E(X,Y)$, which yields
$$c^2(E(X)\otimes E(Y))=c(E(Y)\otimes E(X))+\frac{1}{Y-X}c(E(X,Y)\otimes F(X,Y)+F(X,Y)\otimes E(X,Y))=$$
$$=E(X)\otimes E(Y)+\frac{1}{X-Y}(E(X,Y)\otimes F(X,Y)+F(X,Y)\otimes E(X,Y))+$$
$$+\frac{1}{Y-X}(E(X,Y)\otimes F(X,Y)+F(X,Y)\otimes E(X,Y))=E(X)\otimes F(Y),$$
proving $c^2={\rm id}$. The defining relations of ${\rm Sym}^c(V)$ are then given by
$$[E(X),E(Y)]=\frac{1}{Y-X}(E(X,Y)\cdot F(X,Y)+F(X,Y)\cdot E(X,Y)),$$
$$[E(X),F(Y)]=\frac{1}{Y-X}F(X,Y)^2,\; [F(X),F(Y)]=0.$$
Comparing this with the defining relations of $A$, it suffices to prove that $E(X,Y)$ and $F(X,Y)$ commute in $A$. But this follows from the identity
$$[E(Y),F(X)]=(X-Y)F(X,Y)^2=-(Y-X)F(X,Y)^2=-[E(X),F(Y)]$$
in $A$, finishing the proof.

\begin{conj} The operator $c$ given above satisfies the Quantum Yang-Baxter equation. \end{conj}

\subsection{Differential operators}

We consider the algebra $B=\mathbb{Q}[w_1,w_2,\ldots]$ and define operators $e_i$ and $f_i$ on $B$ as follows: the operator $f_i$ acts by multiplication with $w_i$, and the operator $e_i$ is a derivation defined on generators of $B$ by
$$e_iw_{j+1}=w_{i+1}w_k+\ldots-w_{i+j}w_1-(w_{j+1}w_i+\ldots+w_{i+j}w_1).$$
In more compact form, define
$$E(X)=\sum_{n\geq 0}e_nX^n,\;\;\; F(X)=\sum_{n\geq 0}f_{n+1}X^n,\;\;\; W(X)=\sum_{n\geq 0}w_{n+1}X^n.$$
Then $F(X)$ acts on $B[X]$ via multiplication by $W(X)$, and $E(X)$ acts via derivations such that
$$E(X)W(Y)=(Y-X)(\frac{(YW(Y)-XW(X))}{Y-X})^2.$$

\begin{thm} The algebra $A$ acts on $B$ via the operators defined above.
\end{thm}

\proof We verify that the above operators on $B$ fulfill the defining relations of $A$. Since $F(X)$ acts via multiplication on the commutative algebra $B$, the relation $[F(X),F(Y)]=0$ is obvious.

As in the previous proof, we will abbreviate $W(X,Y)=YW(Y)-XW(X)$, and similarly for $E(X,Y)$, $F(X,Y)$. Since $E(X)$ acts by derivations, we have
$$E(X)(F(Y)(W(Z)))=E(X)(W(Y)W(Z))=E(X)(W(Y))W(Z)+W(Y)E(X)(W(Z))=$$
$$=E(X)(W(Y))W(Z)+F(Y)(E(X)(W(Z))),$$
and thus
$$[E(X),F(Y)](W(Z))=E(X)(W(Y))W(Z)=\frac{1}{Y-X}W(X,Y)^2W(Z)=({(Y-X)}F(X,Y)^2)(W(Z)),$$
verifying the second defining relation. 

The verification of the first defining relation is more involved. We have
$$E(X)(E(Y)(W(Z)))=\frac{1}{Z-Y}E(X)(W(Y,Z)^2)=\frac{2}{Z-Y}W(Y,Z)E(X)(W(Y,Z)),$$
and similarly
$$E(Y)(E(X)(W(Z)))=\frac{2}{Z-X}W(X,Z)E(Y)(W(X,Y)),$$
thus $[E(X),E(Y)](W(Z))=$
$$=\frac{2}{(Z-Y)(Z-X)}((Z-X)W(Y,Z)E(X)(W(Y,Z))-(Z-Y)W(X,Z)E(Y)(W(X,Z)))=$$
$$=\frac{2}{(Z-Y)(Z-X)}(Z(Z-X)W(Y,Z)E(X)(W(Z))-Y(Z-X)W(Y,Z)E(X)(W(Y))-$$
$$-Z(Z-Y)W(X,Z)E(Y)(W(Z))+X(Z-Y)W(X,Z)E(Y)(W(X)))=$$
$$=\frac{2}{(Z-Y)(Z-X)}(ZW(Y,Z)W(X,Z)^2-Y(Z-X)W(Y,Z)E(X)(W(Y))-$$
$$-ZW(X,Z)W(Y,Z)^2+X(Z-Y)W(X,Z)E(Y)(W(X))).$$
Using $W(X,Z)-W(Y,Z)=W(X,Y)$ and $W(Y,X)=-W(X,Y)$, this can be rewritten as
$$\frac{2}{(Z-Y)(Z-X)(Y-X)}(Z(Y-X)W(X,Z)W(Y,Z)W(X,Y)-$$
$$-W(X,Y)^2(Y(Z-X)W(Y,Z)+X(Z-Y)W(X,Z)))=$$
$$=\frac{2}{Y-X}W(X,Y)\frac{1}{(Z-Y)(Z-X)}(Z(Y-X)W(X,Z)W(Y,Z)-$$
$$-W(X,Y)(Y(Z-X)W(Y,Z)+X(Z-Y)W(X,Z))).$$
Again using $W(X,Z)-W(Y,Z)=W(X,Y)$, we find
$$Z(Y-X)W(X,Z)W(Y,Z)-W(X,Y)(Y(Z-X)W(Y,Z)+X(Z-Y)W(X,Z))=$$
$$=Y(Z-X)W(Y,Z)^2-X(Z-Y)W(X,Z)^2,$$
and thus
$$E(X)(E(Y)(W(Z)))=\frac{2}{Y-X}W(X,Y)\frac{Y(Z-X)W(Y,Z)^2-X(Z-Y)W(X,Z)^2}{(Z-Y)(Z-X)}=$$
$$=\frac{2}{Y-X}W(X,Y)(YE(Y)(W(Z))-XE(X)(W(Z)))=\frac{2}{Y-X}F(X,Y)(E(X,Y)(W(Z))),$$
proving the identity $[E(X),E(Y)]=2\frac{1}{Y-X}F(X,Y)E(X,Y)$, which by the proof of Theorem \ref{symc} is equivalent to the first defining relation of $A$. This finishes the proof.

\begin{conj} The subalgebra generated by the operators $e_i$ and $f_i$ inside the ring of differential operators on $B$ is isomorphic to $A$.
\end{conj}

\subsection{Yangian} The category of regular representations of the Kronecker quiver and the category of torsion sheaves on the projective line being equivalent, one can expect an isomorphism of the associated cohomological Hall algebras. In the latter case, this is a special case of the class of algebras which will be described in \cite{SV:19}. The defining relations given there show that $A$ can be viewed as a kind of Yangian algebra.

\subsection*{Acknowledgments} The authors would like to thank Ben Davison, Evgeny Feigin, Catharina Stroppel and Oksana Yakimova for helpful discussions on the results of this work, and Olivier Schiffmann for pointing out \cite{SV:19} and explaining the results obtained there.

\bibliographystyle{abbrv}
\bibliography{Literature}

\end{document}